\documentclass{amsart}

\usepackage{appendix}
\usepackage{amsmath}
\usepackage{amsthm}
\usepackage{amssymb}
\usepackage{amsfonts}
\usepackage{bm}
\usepackage[shortlabels]{enumitem}
\usepackage[bookmarks=true,hyperindex,pdftex,colorlinks,citecolor=red, linkcolor=cyan]{hyperref}
\usepackage{centernot} 
\usepackage{marginnote} 
\usepackage{bbm}
\usepackage[all]{xy}
\usepackage{tikz}
\usepackage{comment}
\usetikzlibrary{arrows}
\usetikzlibrary{shapes,decorations}
\usetikzlibrary{positioning}
\usepackage{xcolor}

\usepackage{mathrsfs}  

\DeclareMathOperator{\Lip}{Lip}                             
\DeclareMathOperator{\lip}{lip}                             

\newcommand{\N}{\mathbb{N}}             
\newcommand{\R}{\mathbb{R}}             

\newcommand{\sequence}[1]{(#1)_{n}}
\newcommand{\iten}{\ensuremath{\widehat{\otimes}_\varepsilon}} 
\newcommand{\pten}{\ensuremath{\widehat{\otimes}_\pi}}         

\newcommand{\F}{\mathcal{F}}                                
\newcommand{\Free}{\mathcal{F}}                             

\newcommand{\la}{\langle}
\newcommand{\ra}{\rangle}
\newcommand{\ep}{\varepsilon}

\theoremstyle{plain}
\newtheorem{theorem}{Theorem}[section]
\newtheorem{lemma}[theorem]{Lemma}
\newtheorem{corollary}[theorem]{Corollary}
\newtheorem{proposition}[theorem]{Proposition}

\theoremstyle{definition}
\newtheorem*{definition*}{Definition}
\newtheorem{definition}[theorem]{Definition}

\theoremstyle{remark}

\begin{document}
\title[The set of elementary tensors is $w$-closed in proj. tensor products]{The set of elementary tensors is weakly closed in projective tensor products.}






\author[C. Petitjean]{Colin Petitjean}
\address[C. Petitjean]{Univ Gustave Eiffel, Univ Paris Est Creteil, CNRS, LAMA UMR8050, F-77447 Marne-la-Vallée, France}
\email{colin.petitjean@univ-eiffel.fr}

\begin{abstract}
    In this short note, we prove that the set of elementary tensors is weakly closed in the projective tensor product of two Banach spaces. As a result, we are able to answer a question from the literature proving that if $(x_n) \subset X$ and $(y_n) \subset Y$ are two weakly null sequences such that $(x_n \otimes y_n)$ converges weakly in $X \pten Y$, then $(x_n \otimes y_n)$ is also weakly null. 
\end{abstract}

\maketitle

\section{Weak convergence in projective tensor product}

Let $X$, $Y$ and $Z$ be real Banach spaces. We denote $\mathcal{B}(X \times Y , Z)$ the space of continuous bilinear operators from $X \times Y$ into $Z$. If $Z = \R$, we simply write $\mathcal{B}(X \times Y)$.
For $x \in X$ and $y \in Y$, define the \textit{elementary tensor} $x\otimes y \in \mathcal{B}(X \times Y )^*$ by: 
$$  \forall B \in \mathcal{B}(X \times Y), \quad  \langle x\otimes y , B \rangle = B(x,y).$$
We then introduce $X \otimes Y := \mathrm{span} \{ x\otimes y \, : \, x \in X, \, y \in Y\}$. Recall that the norm on $\mathcal{B}(X \times Y)$ is defined by $\| B \|_{\mathcal{B}(X \times Y)} = \sup_{x \in B_X, y \in B_Y} |B(x,y)|$. Let $\|\cdot\|_{\pi}$ be the dual norm of $\| \cdot\|_{\mathcal{B}(X \times Y)}$. It is well known (see e.g. \cite[Proposition~VIII.~9.~a)]{vectormeasures}) that if $u \in X \otimes Y$ then $$ \| u \|_{\pi} = \inf \Big \{ \sum_{i=1}^n \|x_i\|  \|y_i \| \, : \, u = \sum_{i=1}^n x_i \otimes y_i \Big \}.$$
The \textit{projective tensor product} of $X$ and $Y$ is defined as follows:
$$X \pten Y = \overline{\mathrm{span}}^{\|\cdot\|_{\pi}} \{ x\otimes y \, : \, x \in X, \, y \in Y\} \subseteq \mathcal{B}(X \times Y)^* .$$
As a consequence of the fundamental linearisation property of tensor products, one easily deduce the following isometric identification $(X \pten Y)^* \equiv \mathcal{B}(X \times Y)$.
Since $\mathcal{B}(X \times Y) \equiv \mathcal{L}(X, Y^*)$, where $\mathcal{L}(X, Y^*)$ stands for the space of bounded linear operators from $X$ to $Y^*$, one also has that $\mathcal{L}(X, Y^*) \equiv (X \pten Y)^*$.
\medskip

The aim of this short note is to answer Question~3.9 in \cite{RR_23}:
\begin{center}
\textit{Let $X$ and $Y$ be Banach spaces. Let $(x_n)_{n\in \N}$ and $(y_n)_{n\in \N}$ be weakly null sequences in $X$ and $Y$, respectively, such that $(x_n \otimes y_n)_{n \in \N}$ is weakly convergent in $X \pten Y$. Is $(x_n \otimes y_n)_{n \in \N}$ weakly null in $X\pten Y$?}
\end{center}
Let 
\begin{equation} \label{eqT}
    \mathcal T= \{ x \otimes y \, : \, x \in X, \, y \in Y \}     
\end{equation}
be the set of elementary tensors in $X \pten Y$. We shall start with a simple but key observation. Recall that a Banach space $X$ has the approximation property (AP in short) if for every $\ep >0$, for every compact subset $K \subset X$, there exists a finite rank operator $T\in \mathcal L(X,X)$ such that $\|Tx -x\| \leq \ep$ for every $x \in K$.

\begin{lemma} \label{lem:matrix}
    Let $X,Y$ be Banach spaces such that $X$ or $Y$ has the AP. Let $T \in X \pten Y$. Then $T \in \mathcal T$  if and only if for every  linearly independent families $\{x_1^*,x_2^* \} \subset X^*$ and $\{y_1^{*},y_2^{*} \} \subset Y^{*}$ we have: 
 $$(\star) \qquad \begin{vmatrix}
 \la  T , x_1^* \otimes y_1^{*} \ra  &  \la  T , x_1^* \otimes y_2^{*} \ra  \\ 
 \la  T , x_2^* \otimes y_1^{*} \ra  &  \la  T , x_2^* \otimes y_2^{*} \ra  
 \end{vmatrix} = 0 .$$
\end{lemma}

\begin{proof}
Thanks to \cite[Proposition 2.8]{ryan}, every $T\in X \pten Y$ can be written as 
$$T = \sum_{n=1}^{\infty} x_n \otimes y_n$$ with $\sum_{n=1}^{\infty} \| x_n \|  \| y_n \| \leq 2 \|T\|$. Moreover, the linear map $\Phi :  X \pten Y \to \mathcal L(X^*,Y)$ obtained by
$$\forall x^* \in X^*, \quad \Phi\Big( \sum_{n=1}^{\infty} x_n \otimes y_n\Big)(x^*) = \sum_{n=1}^{\infty} x^*(x_n) y_n$$ 
defines a bounded operator. Since $X$ or $Y$ has the AP, $\Phi$ is moreover injective; see \cite[Proposition~4.6]{ryan}.

If $T = x \otimes y \in \mathcal T$, then it is straightforward to check that condition $(\star)$ is verified:
 $$ \begin{vmatrix}
 \la  T , x_1^* \otimes y_1^{*} \ra  &  \la  T , x_1^* \otimes y_2^{*} \ra  \\ 
 \la  T , x_2^* \otimes y_1^{*} \ra  &  \la  T , x_2^* \otimes y_2^{*} \ra  
 \end{vmatrix} = \begin{vmatrix}
 x_1^*(x) y_1^{*}(y)   & x_1^*(x) y_2^{*}(y)  \\ 
 x_2^*(x) y_1^{*}(y)  &  x_2^*(x) y_2^{*}(y)
 \end{vmatrix} =  0.$$
Assume now that $T \not \in \mathcal T$. Then $\Phi(T)$ is an operator of rank greater than 2 in $\mathcal L(X^*,Y)$. Thus, there exists a linearly independent family $\{x_1^*,x_2^* \} \subset X^*$ such that $\Phi(T)(x_1^*) \ne 0$, $\Phi(T)(x_2^*) \ne 0$ and $\{\Phi(T)(x_1^*),\Phi(T)(x_2^*) \} \subset Y$ is a linearly independent family. To finish the proof, simply pick a linearly independent family $\{y_1^{*},y_2^{*} \} \subset Y^{*}$ satisfying: 
\begin{center}
\begin{tabular}{ c c }
 $\la \Phi(T)(x_1^*) , y_1^* \ra \ne 0$ & $\la \Phi(T)(x_1^*) , y_2^* \ra =  0$  \\ 
$ \la \Phi(T)(x_2^*) , y_1^* \ra =  0$ & $\la \Phi(T)(x_2^*) , y_2^* \ra \ne 0.$
\end{tabular}
\end{center}
\end{proof}

\begin{proposition} \label{propT}
Let $X,Y$ be two Banach spaces such that $X$ or $Y$ has the AP. Then the set of elementary tensors $\mathcal T$ is weakly closed in $X \pten Y$.
\end{proposition}

\begin{proof} We let $I$ be the set of all vectors $(x_1^*,x_2^* ,y_1^{*},y_2^{*})$ such that $\{x_1^*,x_2^* \} \subset X^*$ and $\{y_1^{*},y_2^{*} \} \subset Y^{*}$ are both linearly independent families. Next, for every $T \in X \pten Y$ and $S=(x_1^*,x_2^* ,y_1^{*},y_2^{*}) \in I$, we define
$$D_S(T) = \begin{vmatrix}
 \la  T , x_1^* \otimes y_1^{*} \ra  &  \la  T , x_1^* \otimes y_2^{*} \ra  \\ 
 \la  T , x_2^* \otimes y_1^{*} \ra  &  \la  T , x_2^* \otimes y_2^{*} \ra  
 \end{vmatrix}.$$
The result now directly follows from Lemma~\ref{lem:matrix} together with the fact that $D_S$ is continuous with respect to the weak topology. Indeed, one can write $\mathcal T$ as an intersection of weakly closed sets:
$$ \mathcal{T} =\bigcap_{S \in I} D_S^{-1}(\{0\}). $$
\end{proof}

The next corollary answers \cite[Question 3.9]{RR_23} positively under rather general assumptions. 

\begin{corollary} \label{cor:wnull}
    Let $X$ and $Y$ be Banach spaces such that $X$ or $Y$ has the AP. 
    If $(x_n)_{n\in \N} \subset X$ converges weakly to $x$, $(y_n)_{n\in \N} \subset Y$ converges weakly to $y$, and $(x_n \otimes y_n)_{n \in \N}$ is weakly convergent in $X \pten Y$, then $(x_n \otimes y_n)_{n \in \N}$ converges weakly to $x \otimes y$.
\end{corollary}

Before proving this corollary, let us point out that the canonical basis $(e_n)_{n \in \N}$ of $\ell_2$ shows that if $(x_n)_{n\in \N}\subset X$ and $(y_n)_{n \in \N} \subset Y$ are weakly null sequences, the sequence $(x_n \otimes y_n)_{n \in \N}$ may fail to be weakly null in $X \pten Y$. Indeed, $(e_n \otimes e_n)_{n \in \N}$ is isometric to the $\ell_1$-canonical basis; see \cite[Example~2.10]{ryan}.

\begin{proof} 
    Assume first that $(x_n)_{n\in \N} \subset X$ and $(y_n)_{n\in \N} \subset Y$ are weakly null sequences such that $(x_n \otimes y_n)_{n \in \N}$ is weakly convergent in $X \pten Y$.
    Since $\mathcal T$ is weakly closed, there exists $x \in X$ and $y\in Y$ such that $x_n \otimes y_n \to x \otimes y$ in the weak topology. 
    Arguing by contradiction, suppose that $x \otimes y \neq 0$. Pick $x^* \in X^*$ and $y^* \in Y^*$ such that 
    $x^*(x) = \|x\| \neq 0$ and $y^*(y) = \|y\| \neq 0$. On the one hand, $x_n \otimes y_n \to x \otimes y$ weakly, so that 
    $$\la x^* \otimes y^* , x_n \otimes y_n \ra \to \la x^* \otimes y^* , x \otimes y \ra = x^*(x)y^*(y) =  \|x\|\|y\| \neq 0.$$
    On the other hand, since $(x_n)_{n\in \N}$ and $(y_n)_{n\in \N} $ are weakly null, one readily obtains a contradiction:
    $$\la x^* \otimes y^* , x_n \otimes y_n \ra = x^*(x_n) y^*(y_n) \to  0.$$

    Similarly, if $(x_n)_{n\in \N} \subset X$ converges weakly to $x$, $(y_n)_{n\in \N} \subset Y$ converges weakly to $y$, and $(x_n \otimes y_n)_{n \in \N}$ is weakly convergent in $X \pten Y$, then we write:
    $$(x-x_n)\otimes (y-y_n) = x \otimes y - x \otimes y_n - x_n \otimes y + x_n \otimes y_n.$$
    But, $x \otimes y_n \overset{w}{\underset{n\to +\infty}{\longrightarrow}} x\otimes y$ and $x_n \otimes y \overset{w}{\underset{n\to +\infty}{\longrightarrow}} x\otimes y$. Therefore $\big((x-x_n)\otimes (y-y_n)\big)_{n \in \N}$ converges weakly and moreover the weak limit must be 0 thanks to the first part of the proof. This implies that $x_n \otimes y_n \overset{w}{\underset{n\to +\infty}{\longrightarrow}} x\otimes y$.
\end{proof}

In connection with Proposition~\ref{propT}, we also wish to mention \cite[Theorem~2.3]{GGMR23} which we describe now. If $C$ and $D$ are subsets of $X$ and $Y$ respectively, then let 
$$C \otimes D :=\{x \otimes y \; : \; x \in C, \; y \in D \} \subset \mathcal T.$$
Theorem~2.3 in \cite{GGMR23} states that if $C$ and $D$ are bounded then $\overline{C}^w \otimes \overline{D}^w = \overline{C \otimes D}^w$ in $X \pten Y$. The technique which we introduced in the present note permits to remove the boundedness assumption in the particular case when $C$ and $D$ are subspaces. It also allows us to slighly simplify the original proof of \cite[Theorem~2.3]{GGMR23}. The next lemma is the main ingredient. 

\begin{lemma} \label{lemma2}
    Let $X$ and $Y$ be Banach spaces such that $X$ or $Y$ has the AP.\\ Let $(x_s)_s \subset X$ and $(y_s)_s \subset Y$ be two nets such that $x_s \to x^{**}$ in the weak$^*$-topology of $X^{**}$, $y_s \to y^{**}$ in the weak$^*$-topology of $Y^{**}$, and $(x_s\otimes y_s)_s$ converges in the weak$^*$-topology of $(X \pten Y)^{**}$. Then $(x_s\otimes y_s)_s$ converges weakly$^*$ to $x^{**} \otimes y^{**}$.
\end{lemma}

The proof is essentially the same as that of Corollary~\ref{cor:wnull}, so we leave the details to the reader.

\begin{corollary} \label{Cor-WeakClosure}
    Let $X$ and $Y$ be Banach spaces such that $X$ or $Y$ has the AP.\\ 
     If $C$ and $D$ are subsets of $X$ and $Y$ respectively, then $\overline{C}^w \otimes \overline{D}^w = \overline{C \otimes D}^w$ if one of the following additional assumptions are satisfied: 
     \begin{enumerate}[$(i)$]
         \item If $C$ and $D$ are subspaces.
         \item If $C$ and $D$ are bounded. 
     \end{enumerate}
\end{corollary}
    
\begin{proof}
    First of all, it is readily seen that one has $\overline{C}^w \otimes \overline{D}^w \subset  \overline{C \otimes D}^w$ without any additional assumption on $C$ and $D$ (see the first part of the proof of \cite[Theorem~2.3]{GGMR23}). Therefore we only have to prove the reverse inclusion in both cases. 

    To prove $(i)$, it suffices to apply Proposition~\ref{propT}: 
    $$C \otimes D \subset \overline{C} \otimes \overline{D} \implies \overline{C \otimes D}^w \subset \overline{\overline{C} \otimes \overline{D}}^w = \overline{C} \otimes \overline{D}.    $$ 

    To prove $(ii)$, let $z \in  \overline{C \otimes D}^w$. We fix a net $(x_s \otimes y_s)_s\subset C \otimes D$ which converges weakly to $z$. Thanks to Proposition~\ref{propT}, there exist $x \in X$ and $y \in Y$ such that $z = x \otimes y$. Since $C$ and $D$ are bounded, up to taking a suitable subnet, we may assume that both $x_s \to x^{**}$ in the weak$^*$-topology of $X^{**}$ and $y_s \to y^{**}$ in the weak$^*$-topology of $Y^{**}$. Thanks to Lemma~\ref{lemma2}, $x_s \otimes y_s \to x^{**} \otimes y^{**}$ in the weak$^*$-topology of $(X \pten Y)^{**}$. By uniqueness of the limit, $x^{**} \otimes  y^{**} = z = x \otimes y$. We distinguish two cases.

    If $z = 0$ then $x^{**} = 0$ or $y^{**} = 0$. Say $x^{**} = 0$ for instance. This means that $0 \in \overline{C}^w$. Now pick any $y \in C$ and observe that $z = 0 \otimes y$, which was to be shown.

    If $z \neq 0$, then it is readily seen that $x^{**} \in \mathrm{span}\{x\}$ and $y^{**} \in \mathrm{span}\{y\}$. Therefore $x^{**} \in \overline{C}^{w^*} \cap X = \overline{C}^w$ and $y^{**} \in \overline{D}^{w^*} \cap Y = \overline{D}^w$, which concludes the proof.
\end{proof}

\section{Applications to vector-valued Lipschitz free spaces}

If $M$ is a pointed metric space,  with base point $0 \in M$, and if $X$ is a real Banach space, then $\Lip_0(M,X)$ stands for the vector space of all Lipschitz maps from $M$ to $X$ which satisfy $f(0)=0$. Equipped with the Lipschitz norm:
$$ \forall f \in \Lip_0(M,X), \quad \|f\|_L = \sup_{x \neq y \in M} \frac{\|f(x)-f(y)\|_X}{d(x,y)},$$
$\Lip_0(M,X)$ naturally becomes a Banach space. When $X = \R$, it is customary to omit the reference to $X$, that is $\Lip_0(M):=\Lip_0(M,\R)$.
Next, for $x\in M$, we let $\delta(x) \in \Lip_0(M)^*$ be the evaluation functional defined by $\langle\delta(x) , f \rangle  = f(x), \ \forall f\in \Lip_0(M).$ The Lipschitz free space over $M$ is the Banach space
    $$\F(M) := \overline{ \mbox{span}}^{\| \cdot  \|}\left \{ \delta(x) \, : \, x \in M  \right \} \subset \Lip_0(M)^*.$$
The universal extension property of Lipschitz free spaces states that for every $f \in \Lip_{0}(M,X)$, there exists a unique continuous linear operator $\overline{f} \in \mathcal{L}(\F(M),X)$ such that:
\begin{enumerate}[label = (\roman*)]
\item $f=\overline{f} \circ \delta$, and
\item $\| \overline{f} \|_{\mathcal{L}(\F(M),X)} = \| f \|_L$.
\end{enumerate}
In particular, the next isometric identification holds: 
$$\Lip_0(M,X) \equiv \mathcal{L}(\F(M),X).$$
A direct application (in the case $X = \R$) provides another basic yet important information: 
$$\Lip_0(M) \equiv \F(M)^*.$$
It also follows from basic tensor product theory that $\Lip_0(M,X^*) \equiv (\F(M) \pten X)^*$, which leads to the next definition (see \cite{vectorvalued} for more details):

\begin{definition}[Vector-valued Lispschitz free spaces]
 Let $M$ be a pointed metric space and let $X$ be a Banach space. We define the $X$-valued Lipschitz free space over $M$ to be: $\F(M,X) :=  \F(M) \pten X$.
\end{definition}

\subsection{Weak closure of \texorpdfstring{$\delta(M,X)$}{delta(M,X)}.}

It is proved in \cite[Proposition~2.9]{GPPR18} that $\delta(M) = \{\delta(x) : x \in M \}$ is weakly closed in $\F(M)$ provided that $M$ is complete. Our first aim is to prove the vector-valued counterpart. For this purpose, we need to identify a set that corresponds to $\delta(M)$ in the vector-valued case. A legitimate set to look at is the following:
$$\delta(M,X)  := \{ \delta(y) \otimes x \; : \; y \in M, \, x \in X \} \subset \F(M,X).$$
Notice that this does not exactly correspond to $\delta(M)$ in the case $X = \R$ since we have $\delta(M,\R) = \R \cdot \delta(M)$. This discrepancy is not a major issue since $\R \cdot \delta(M)$ is also a weakly closed set when $M$ is complete. The next result is thus a natural extension to the vector valued setting of \cite[Proposition~2.9]{GPPR18}.

\begin{proposition} \label{prop:deltavectorclosed}
Let $M$ be a complete pointed metric space and $X$ be a Banach space such that $\F(M)$ or $X$ have the approximation property. Then $\delta(M,X)$ is weakly closed in $\Free(M,X)$.
\end{proposition}

\begin{proof}
In what follows, $\mathcal T$ denotes the elementary tensors in $\F(M) \pten X$. Consider a net $(\delta(m_{\alpha})\otimes x_\alpha)_\alpha \subset \delta(M,X)$ which is weakly convergent. Since $\delta(M,X) \subset \mathcal T$ and $\mathcal T$ is weakly closed (Proposition~\ref{propT}), there exist $\gamma \in \F(M)$ and $x\in X$ such that the net goes to $\gamma \otimes x$ in the weak topology. We may assume that $x \neq 0$, otherwise there is nothing to do. Pick $x^* \in X^*$ such that $x^*(x)\neq 0$. Then, for every $f \in \Lip_0(M)$, we have that $f(m_\alpha)x^*(x_\alpha) \to f(\gamma) x^*(x)$. So the net $\big(\frac{x^*(x_\alpha)}{x^*(x)} \delta(m_\alpha)\big)_\alpha \subset \R \cdot \delta(M)$ weakly converges to $\gamma$. Since $\R \cdot \delta(M)$ is weakly closed, there is $\lambda \in \R$ and $m \in M$ such that $\gamma = \lambda \delta(m)$. Consequently $\gamma \otimes x = \delta(m) \otimes \lambda x \in \delta(M,X)$.
\end{proof}

\subsection{Natural preduals}
Next, following \cite[Section 3]{GPPR18}, $S \subset \Lip_0(M)$ is a natural predual of $\F(M)$ if $S^* \equiv \F(M)$ and $\delta(B(0,r))$ is $\sigma(\F(M), S)$-closed for every $r\geq 0$. A reasonable extension of this notion in the vector-valued setting is the following. 

\begin{definition}\label{def:naturalvector}
Let $M$ be a pointed metric space and $X$ be a Banach space with $\dim(X)\geq 2$. We say that a Banach space $S$ is a natural predual of $\Free(M,X^*)$ if $Y^*\equiv \Free(M,X^*)$  and 
$$\delta(B(0,r),X^*) = \{\delta(m) \otimes x^* \, : \, m \in B(0,r), \, x^* \in X^* \} \subset \F(M,X^*)$$
is $\sigma(\Free(M,X^*),S)$-closed for every $r \geq 0$.
\end{definition}

Notice again that $\delta(B(0,r),\R) = \R \cdot \delta(B(0,r))$. In the next statement, $\lip_0(M)$ denotes the subspace of $\Lip_0(M)$ of all uniformly locally flat functions. Recall that $f \in \Lip_0(M)$ is uniformly locally flat if 
$$ \lim\limits_{d(x,y) \to 0} \frac{|f(x) - f(y)|}{d(x,y)} = 0.$$

\begin{lemma} \label{lemma:weakstarclosed}
Let $M$ be a separable pointed metric space. Suppose that $S\subset \lip_0(M)$ is a natural predual of $\F(M)$.
Then, for every $r \geq0$, $\R \cdot \delta(B(0,r))$ is weak$^*$ closed in $\F(M)$.
\end{lemma}

\begin{proof}
Let us fix $r \geq 0$. Let $\sequence{\lambda_n \delta(x_n)} \subset \R \cdot \delta(B(0,r))$ be a sequence converging to some $\gamma \in \F(M)$ in the  weak$^*$ topology. We assume that $\gamma \neq 0$, otherwise there is nothing to do. Since a weak$^*$ convergent sequence is bounded, and by weak$^*$ lower-semi-continuity of the norm, we may assume that there exists $C>0$ such that for every $n$:
$$ 0 < \frac{\|\gamma\|}{2} \leq |\lambda_n| \| \delta(x_n) \| = |\lambda_n| d(x_n,0) \leq C.$$
Thus, $d(x_n,0) \neq 0$ and $\lambda_n \neq 0$ for every $n$. Up to extracting a further subsequence, we may assume that the sequence $\sequence{\lambda_n d(x_n,0)}$ converges to some $\ell \neq 0$. Since $\sequence{x_n} \subset B(0,r)$, we also assume that $\sequence{d(x_n,0)}$ converges to some $d$. We will distinguish two cases.

If $d \neq 0$, then $\sequence{\lambda_n}$ converges to $\lambda:=\frac{\ell}{d}$ and so $\sequence{\delta(x_n)}$ weak$^*$ converges to $\frac{\gamma}{\lambda}$. Since $S$ is a natural predual of $\F(M)$, $\delta(B(0,r))$ is weak$^*$ closed in $\F(M)$. So there exists $x \in M$ such that $\gamma = \lambda \delta(x)$. 

If $d = 0$, then $\sequence{\delta(x_n)}$ converges to $0$ in the norm topology (and $\sequence{\lambda_n}$ tends to infinity). Note that we may write:
$$ \lambda_n \delta(x_n) =  \lambda_n d(x_n,0) \frac{\delta(x_n) - \delta(0)}{d(x_n,0)}. $$
Since $S\subset \lip_0(M)$, the sequence $\sequence{\frac{\delta(x_n) - \delta(0)}{d(x_n,0)}}$ weak$^*$ converges to 0. Moreover the sequence $\sequence{\lambda_n d(x_n,0)}$ converges to $\ell \neq 0$. Consequently $\sequence{\lambda_n \delta(x_n)}$ weak$^*$ converges to 0 and so $\gamma = 0$, which is a contradiction.
\end{proof}

Before going further, we need to introduce the injective tensor product of two Banach spaces. Recall that, to define the projective tensor product, we introduced $x\otimes y$ as an element of $\mathcal{B}(X \times Y)^{*}$. For the injective tensor product, we change the point of view since we now consider $x\otimes y$ as an element of $\mathcal{B}(X^* \times Y^*)$ defined as follows: 
$$\forall (x^*,y^*) \in X^* \times Y^*,  \quad \la x\otimes y , (x^*,y^*) \ra  = x^*(x) y^*(y).$$
In this case, we denote $\| \cdot \|_{\ep}$ the canonical norm on $\mathcal{B}(X^* \times Y^*)$. Thus, if $u = \sum_{i=1}^n x_i \otimes y_i \in X \otimes Y$ then 
$$\| u \|_{\ep} = \sup \Big \{ \Big| \sum_{i=1}^n  x^*(x_i) y^*(y_i) \Big|  \, : \, x^* \in B_{X^*}, y^* \in B_{Y^*} \Big \}.$$
The \textit{injective tensor product} of $X$ and $Y$ is defined by: 
$$X \iten Y = \overline{\mathrm{span}}^{\|\cdot\|_{\ep}} \{ x\otimes y \, : \, x \in X, \, y \in Y\} \subseteq \mathcal{B}(X^* \times Y^*) .$$
In the sequel, we will use a classical result from tensor product theory (see e.g. \cite[Theorem 5.33]{ryan}):
If $X^*$ or $Y^*$ has the Radon-Nikod\'ym property (RNP in short) and that $X^*$ or $Y^*$ has the AP, then $(X \iten Y)^* \equiv X^* \pten Y^*$. The RNP has many characterizations, we refer to reader to Section~VII.6 in \cite{vectormeasures} for a nice overview.  
\medskip

Assume now that there exists a subspace $S$ of $\Lip_0(M)$ such that $S^* \equiv \F(M)$. Then one has
$$ \mathcal F(M,X^*) = \mathcal F(M) \widehat\otimes_\pi X^* \equiv (S \iten X)^{*}$$
whenever either $\mathcal F(M)$ or $X^*$ has the AP and either $\mathcal F(M)$ or $X^*$ has the RNP. It is quite natural to wonder whether there are conditions which ensure that $S \iten X$ is a natural predual of $\F(M,X^*)$.
The next result asserts that this sometimes relies on the scalar case. 

\begin{proposition}
Let $M$ be a separable pointed metric space, $S\subset \lip_0(M)$ be a natural predual of $\F(M)$ and $X$ be a Banach space (with $\dim(X) \geq 2$).  Assume moreover that either $\mathcal F(M)$ or $X^*$ has the AP and either $\mathcal F(M)$ or $X^*$ has the RNP. Then $S \iten X$ is a natural predual of $\F(M,X^*)$.
\end{proposition}

\begin{proof}
To show that $S \iten X$ is a natural predual, we essentially follow the proof of Proposition~\ref{prop:deltavectorclosed}. 
First of all, we show that $\mathcal T:= \{ \gamma \otimes x^* \, : \, \gamma \in \F(M), \, x \in X^* \}$ is weak$^*$ closed in $\F(M,X^*)$. Indeed, it is not hard to show that if $T \in \Free(M,X^*)$, then $T \in \mathcal T$ if and only if for every  linearly independent families $\{f_1,f_2 \} \subset S$ and $\{x_1,x_2\} \subset X$ we have: 
 $$\begin{vmatrix}
 \la  T , f_1 \otimes x_1 \ra  &  \la  T , f_1 \otimes x_2 \ra  \\ 
 \la  T , f_2 \otimes x_1 \ra  &  \la  T , f_2 \otimes x_2 \ra  
 \end{vmatrix} = 0 .$$
Accordingly, $\mathcal T$ is weak$^*$ closed. Now we fix $r>0$. Let us consider a net $(\delta(m_{\alpha})\otimes x_\alpha^*)_\alpha \subset \delta(B(0,r),X^*)$ which weak$^*$ converges to some $\gamma \otimes x^* \in \mathcal T$. We may assume that $x^* \neq 0$ otherwise there is nothing to do. Consider $x \in X$ such that $x^*(x)\neq 0$. Then, for every $f \in S$ we have that $f(m_\alpha)x^*(x_\alpha) \to f(\gamma) x^*(x)$. So the net $\big(\frac{x^*(x_\alpha)}{x^*(x)} \delta(m_\alpha)\big)_\alpha \subset \R \cdot \delta(M)$ weak$^*$ converges to $\gamma$. Since $\R \cdot \delta(M)$ is weak$^*$ closed (Lemma~\ref{lemma:weakstarclosed}), there is $\lambda \in \R$ and $m \in M$ such that $\gamma = \lambda \delta(m)$.
\end{proof}

\section*{Acknowledgments}

The author was partially supported by the French ANR project No. ANR-20-CE40-0006. 
He also thanks Christian Le Merdy and Abraham Rueda Zoca for useful discussions.

\end{document}